\documentclass[11pt]{article}

\usepackage[colorlinks=true, linkcolor=blue, citecolor=blue]{hyperref}
\usepackage{amsmath, amssymb, amsthm}
\usepackage[dvipsnames]{xcolor}
\usepackage{bbm, bm}
\usepackage{graphicx}
\usepackage{geometry}
\usepackage{tikz}
\usepackage{subcaption}
\usepackage{footnote} 
\makesavenoteenv{tabular} 
\makesavenoteenv{table} 

\newcommand{\Av}{\ensuremath{\mathrm{Av}}}
\newcommand{\D}[1]{\ensuremath{\mathrm{\textbf{D}}}}
\renewcommand{\P}{\ensuremath{\mathcal{P}}}

\newcommand{\Q}{\ensuremath{\mathcal{Q}}}

\newcommand{\DD}{\ensuremath{\mathcal{D}}}

\newcommand{\fL}{\ensuremath{\mathsf{L}}}
\newcommand{\fR}{\ensuremath{\mathsf{R}}}
\newcommand{\fD}{\ensuremath{\mathsf{D}}}
\newcommand{\fU}{\ensuremath{\mathsf{U}}}
\newcommand{\fC}{\ensuremath{\mathsf{C}}}
\newcommand{\fB}{\ensuremath{\mathsf{B}}}
\newcommand{\fS}{\ensuremath{\mathsf{S}}}
\newcommand{\fM}{\ensuremath{\mathsf{M}}}

\newcommand{\fX}{\ensuremath{\mathsf{X}}}

\newcommand{\bD}{\ensuremath{\textbf{\textsf{D}}}}
\newcommand{\bU}{\ensuremath{\textbf{\textsf{U}}}}
\newcommand{\bC}{\ensuremath{\textbf{\textsf{C}}}}
\newcommand{\bE}{\ensuremath{\textbf{\textsf{E}}}}

\newcommand{\bc}{\ensuremath{\mathbf{c}}}
\newcommand{\be}{\ensuremath{\mathbf{e}}}

\definecolor{keywords}{RGB}{255,0,90}
\definecolor{comments}{RGB}{0,0,113}
\definecolor{myred}{RGB}{160,0,0}
\definecolor{green}{RGB}{0,150,0}

\newcommand\dyck[4]{
  \fill[white]  (#1) rectangle +(#2,#2);
  \draw[help lines] (#1) grid +(#2,#2);
  \coordinate (prev) at (#1);
  \foreach \dir in {#4}{
    \ifnum\dir=0
    \coordinate (dep) at (1,0);
    \else
    \coordinate (dep) at (0,1);
    \fi
    \draw[line width=2pt, color=#3, cap=round] (prev) -- ++(dep) coordinate (prev);
  };
}

\newcommand\pdyck[3]{
  \coordinate (prev) at (#1);
  \foreach \dir in {#3}{
    \ifnum\dir=0
    \coordinate (dep) at (1,0);
    \else
    \coordinate (dep) at (0,1);
    \fi
    \draw[line width=2pt, color=#2, cap=round, style=dotted] (prev) -- ++(dep) coordinate (prev);
  };
}

\newtheorem{theorem}{Theorem}
\newtheorem{lemma}{Lemma}
\theoremstyle{definition}

\newtheorem{example}{Example}

\begin{document}

\newsavebox{\smlmata}
\savebox{\smlmata}{$\left(\begin{smallmatrix} \emptyset & \Av(12)\\ \Av(21) & \emptyset \end{smallmatrix}\right)$}
\newsavebox{\smlmatb}
\savebox{\smlmatb}{$\left(\begin{smallmatrix}\Av(21) & \Av(12) \end{smallmatrix}\right)$}

\title{Juxtaposing Catalan permutation classes with monotone ones}
\author{R. Brignall\thanks{School of Mathematics and Statistics, The Open University, Milton Keynes, MK7 6AA, UK}\\ \small{\texttt{rbrignall@gmail.com}} \and J. Slia\v{c}an\footnotemark[1]\\ \small{\texttt{jakub.sliacan@gmail.com}}}
\date{}

\maketitle

\begin{abstract}
This paper enumerates all juxtaposition classes of the form ``$\Av(abc)$ next to $\Av(xy)$'', where $abc$ is a permutation of length three and $xy$ is a permutation of length two. We use Dyck paths decorated by sequences of points to represent elements from such a juxtaposition class. Context free grammars are then used to enumerate these decorated Dyck paths.
\end{abstract}

\setcounter{tocdepth}{1}

\section{Introduction}
Juxtapositions are a simple special case of permutation grid classes. Grid classes found application mainly as tools to study the structure of other permutation classes. In~\cite{vatter11small}, Vatter used the structural insights that monotone grid classes offer to classify growth rates of small permutation classes. Apart from enumerating permutation classes, several other applications of grid classes exist, among them~\cite{albert2011enumeration, aabrv2013, bevan-new, brignall2012pwo, murphy2003pwo, vatter2011pwo}. For an introduction to grid classes and their further uses, see Bevan's PhD thesis~\cite{bevan2015thesis}.

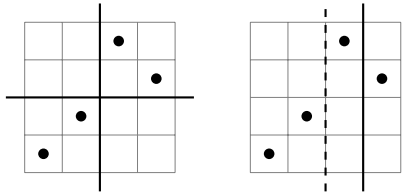
\begin{figure}[!ht]
\begin{center}
\captionsetup{singlelinecheck=off}
  \begin{tikzpicture}[baseline=-0ex,scale=0.5]
  \tikzstyle{vertex}=[circle,fill=black, minimum size=4pt,inner sep=1pt]
  \node[vertex] (v1) at (0.5, 0.5){};
  \node[vertex] (v2) at (1.5,1.5){};
  \node[vertex] (v3) at (2.5, 3.5){};
  \node[vertex] (v4) at (3.5, 2.5){};
  \draw[gray, very thin] (0,0) grid (4,4);
  \draw[thick] (2,-0.5)--(2,4.5);
  \draw[thick] (-0.5,2)--(4.5,2);
  \draw (v1) (v2) (v3) (v4);

  \node[vertex] (v5) at (6.5, 0.5){};
  \node[vertex] (v6) at (7.5, 1.5){};
  \node[vertex] (v7) at (8.5, 3.5){};
  \node[vertex] (v8) at (9.5, 2.5){};
  \draw[gray, very thin] (6,0) grid (10,4);
  \draw[thick, dashed] (8,-0.5)--(8,4.5);
  \draw[thick] (9,-0.5)--(9,4.5);
  \draw (v5) (v6) (v7) (v8);
  \end{tikzpicture}
\caption{\small On the left is the unique gridding of 1243 by the gridding matrix $M =$~\usebox{\smlmata}. On the right are the two griddings of 1243 by $M =$~\usebox{\smlmatb}.}
\label{fig:gridexample}
\end{center}
\end{figure}

Each permutation in a \emph{grid class} can be drawn into a grid so that the subpermutation in each box is in the class specified by the corresponding cell in a \emph{gridding matrix}. See Figure~\ref{fig:gridexample}  for an example of permutations from a \emph{monotone grid class} (where each cell in the gridding matrix is a 21-avoider, 12-avoider, or empty). Because of their more general applicability, the study of grid classes in their own right has emerged in a few directions. For instance, it is conjectured that all monotone grid classes are finitely based, but this is only known for a few special cases, most notably those whose row-column graph is acyclic~\cite{aabrv2013}, and a few other special cases (see~\cite{albert-brignall-2times2, atkinson1997restricted, bevan2015thesis,  waton}). In another direction, the role of grid classes with respect to partial well-ordering has been explored in e.g.~\cite{brignall2012pwo, murphy2003pwo, vatter2011pwo}. Finally, while the asymptotic enumeration of monotone grid classes was answered completely by Bevan~\cite{bevan:growth-rates:}, exact enumeration is harder, primarily due to the difficulty of handling multiple griddings: that is, enumerating `griddable' objects rather than `gridded' ones. One general result here is that all geometric grid classes have rational generating functions~\cite{aabrv2013}, but the move from `gridded' to `griddable' is nonconstructive, instead relying on properties of regular languages.

As a first step towards enumerating more general grid classes, in this paper we replace one cell in the gridding matrix $M$ of a monotone grid class by a \emph{Catalan class}, that is, one avoiding a single permutation of length 3. For simplicity, we restrict our attention to $1\times 2$ grids, although the techniques presented here could be used in larger grids. These $1\times 2$ grid classes are also referred to as \emph{juxtapositions} --- in our case a Catalan class in the left cell, and a monotone class in the right one.

\begin{table}[!ht]
\centering
\begin{tabular}{|c c c|}
\hline
$\Av(213|21)$, \underline{\boldmath{$\Av(231|12)$}} & $\overset{\theta}{\longleftrightarrow}$ & $\Av(123|21)$, \underline{$\Av(321|12)$} \\
$\Av(123|12)$, \underline{\boldmath{$\Av(321|21)$}} & $\overset{\psi}{\longleftrightarrow}$ & $\Av(213|12)$, \underline{$\Av(231|21)$} \\
$\Av(132|12)$, \underline{\boldmath{$\Av(312|21)$}} & $\overset{\phi}{\longleftrightarrow}$ & $\Av(132|21)$, \underline{$\Av(312|12)$} \\
  \hline
\end{tabular}
\caption{\small Each row contains equinumerous classes. Classes in pairs (separated by commas) are equinumerous by symmetry. Left column and right column (of the same row) are equinumerous by one of the bijections $\theta, \psi, \phi$, see Section~\ref{sec:bijections}. Each bijection describes a correspondence between the underlined classes in the given row. The classes in bold are enumerated via context-free grammars.}
\label{tab:bijections}
\end{table}

Let $P_i$ and $S_j$ be permutations, for all $1\leq i\leq k$ and $1 \leq j \leq l$. Also, let $\mathbf{U} = \Av(P_1,\ldots,P_k)$ and $\mathbf{V} = \Av(S_1,\ldots,S_\ell)$ be  classes of permutations that avoid $P_1,\ldots, P_k$ and $S_1,\ldots,S_\ell$, respectively. A \emph{juxtaposition class} $\mathbf{W} = \Av(P_1,\ldots,P_k\mid S_1, \ldots, S_\ell)$ of classes $\mathbf{U}$ and $\mathbf{V}$ is the set of permutations whose form is $AB$ with $A\in \mathbf{U}$ and $B \in \mathbf{V}$. We say that $\mathbf{U}$ is on the left-hand side (LHS) of $\mathbf{W}$ and that $\mathbf{V}$ is on the right-hand side (RHS) of $\mathbf{W}$. The diagram on the right in Figure~\ref{fig:gridexample} shows a permutation 1342 from the juxtaposition class $\Av(21|12)$, together with two possible griddings of 1342. Equipped with the definition of a juxtaposition, we can now refer to Table~\ref{tab:bijections}. It schematizes the relationships between juxtapositions $\Av(P|S)$, where $P$ is of length three and $S$ of length two or vice versa. Out of all these, only twelve are essentially distinct (not symmetries of each other), and all of them have the Catalan class $\Av(abc)$ on the same side in the juxtaposition. So, without loss of generality, we assume a Catalan class to be on the left and a monotone class on the right. By further symmetries, these twelve juxtaposition classes can be coupled into equinumerous pairs. See Table~\ref{tab:bijections}. To the best of our knowledge, only two of the twelve classes in Table~\ref{tab:bijections} have been enumerated --- $\Av(231|12)$ by Bevan~\cite{bevan-new} (and hence $\Av(213|21)$) and $\Av(321|12)$ by Miner~\cite{miner16twobyfour} (and hence $\Av(123|21)$). The goal of this paper is to enumerate the remaining two juxtaposition classes in bold in Table~\ref{tab:bijections} and find bijections between the pairs of underlined classes. This completes the enumeration for all juxtapositions of a Catalan class with a monotone one.

In section~\ref{sec:defs} we introduce the concepts that we need and demonstrate them on an example. Section~\ref{sec:enum} contains enumerations of the boldface classes $\Av(231|12)$, $\Av(321|21)$ and $\Av(312|21)$.  Bijections between the underlined classes are presented in Section~\ref{sec:bijections}. We mention open questions in Section~\ref{sec:conclusion}.

\section{Definitions and overview}
\label{sec:defs}
We assume the reader is familiar with elementary definitions. They are conveniently presented in a note by Bevan~\cite{bevan2015defs}.  We refer to the three juxtaposition classes in bold in Table~\ref{tab:bijections} as $\mathbf{A}, \mathbf{B}$, and $\mathbf{C}$, respectively. We list them below together with their representations in terms of bases, using~\cite{atkinson1997restricted}.

\begin{center}
\begin{tabular}{c c c c l}
$\mathbf{A}$&:= & $\Av(231|12)$ & = & $\Av(2314, 2413, 3412)$ \\
$\mathbf{B}$&:= & $\Av(321|21)$ & = & $\Av(4321, 32154, 42153, 52143, 43152, 53142)$ \\
$\mathbf{C}$&:= & $\Av(312|21)$ & = & $\Av(4132, 4231, 31254, 41253)$\\  
\end{tabular}
\end{center}

\noindent A quick enumeration (using PermLab~\cite{albertpermlab}) of the first twelve elements in $\mathbf{A}, \mathbf{B}$, and $\mathbf{C}$ yields the following sequences.

\begin{center}
\begin{tabular}{l l}
$\mathbf{A}$: \href{http://oeis.org/A033321}{A033321}& 1, 2, 6, 21, 79, 311, 1265, 5275, 22431, 96900, 424068, 1876143\\
$\mathbf{B}$: \href{http://oeis.org/A278301}{A278301}& 1, 2, 6, 23, 98, 434, 1949, 8803, 39888, 181201, 825201, 3767757\\
$\mathbf{C}$: \href{http://oeis.org/A165538}{A165538}& 1, 2, 6, 22, 88, 367, 1568, 6810, 29943, 132958, 595227, 2683373\\
\end{tabular}
\end{center}
We observe that the sequence \href{http://oeis.org/A165538}{A165538} enumerates both classes $\mathbf{C}$ and $\Av(4312, 3142)$. The latter is studied as Example 2 in~\cite{albert2012inflations}. We do not know of a straightforward reason for these to be equinumerous.

A \emph{Dyck path} of length $2n$ is a path on the integer grid from $(0,0)$ to $(n,n)$ where each step is either $(k,m) \to (k+1,m)$ or $(k,m) \to (k,m+1)$. A Dyck path must stay on one side of the diagonal, i.e.~it is not allowed to cross but can \emph{touch} the diagonal. See Figure~\ref{fig:dyckexample} for an example. The orientation of the Dyck path in our definition is arbitrary, and we will use the term ``Dyck path'' to refer to any of the four symmetries of a Dyck path (above/below the diagonal and top-to-bottom/bottom-to-top). In this paper, we only need two kinds of Dyck paths: those that use \emph{up} and \emph{right} steps (bottom-left to top-right Dyck paths that stay above the diagonal) and those that use \emph{down} and \emph{left} steps (top-right to bottom-left Dyck paths that stay below the diagonal). We will always specify which one of the two cases we work with.\\
\begin{figure}[!ht]
\begin{center}
\begin{tikzpicture}[scale=0.3]
\dyck{0,0}{10}{black}{0,0,0,1,1,0,0,1,0,0,1,0,1,1,1,1,0,0,1,1}
  \draw[dashed] (0,0) -- (10,10); 
\end{tikzpicture}
\end{center}
\caption{\small Example of a Dyck path of length 20.}
\label{fig:dyckexample}
\end{figure}
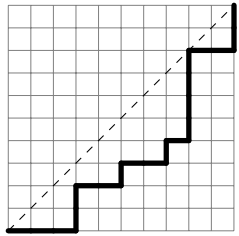

Let $\DD$ be Dyck path of length $2n$ from top right to the bottom left corner of the grid. Let it be below the diagonal. \emph{Offsetting} a diagonal means considering a line with slope one from $(1,0)$ to $(n,n-1)$ on a square grid, instead of the original diagonal $(0,0)$ to $(n,n)$ (in case $\DD$ is above the diagonal, we offset in the opposite direction). Figure~\ref{fig:bij231dyck} illustrates this with the dotted line moving downwards along the $\backslash$-diagonal from 3rd to 4th and from 4th to 5th subfigure. An \emph{excursion} in $\DD$ is the part of the Dyck path $\DD$ which meets the diagonal only at the beginning and at the end.

It is well known that permutations of length $n$ in $\Av(231)$ are in one-to-one correspondence with Dyck paths of length $2n$. And so are permutations of length $n$ in $\Av(321)$. There are a number of bijections between these Catalan objects. We now fix one for each correspondence and describe them in detail.

\begin{figure}[!ht]
\begin{center}
\begin{tikzpicture}[scale=0.25]
\dyck{0,0}{9}{black}{0,0,0,1,1,0,0,1,0,0,1,0,1,1,1,1,0,1}
\draw[dashed] (0,0) -- (9,9); 
\end{tikzpicture}
\begin{tikzpicture}[scale=0.25]
\dyck{0,0}{9}{black}{0,0,0,1,1,0,0,1,0,0,1,0,1,1,1,1,0,1}
\draw[dashed] (0,0) -- (9,9); 
\filldraw[black] (2.5,0.5) circle (6pt);
\filldraw[black] (4.5,2.5) circle (6pt);
\filldraw[black] (6.5,3.5) circle (6pt);
\filldraw[black] (7.5,4.5) circle (6pt);
\filldraw[black] (8.5,8.5) circle (6pt);
\end{tikzpicture}
\begin{tikzpicture}[scale=0.25]
\dyck{0,0}{9}{black}{0,0,0,1,1,0,0,1,0,0,1,0,1,1,1,1,0,1}
\draw[dashed] (0,0) -- (8,8); 
\filldraw[black] (2.5,0.5) circle (6pt);
\filldraw[black] (4.5,2.5) circle (6pt);
\filldraw[black] (6.5,3.5) circle (6pt);
\filldraw[black] (7.5,4.5) circle (6pt);
\filldraw[black] (8.5,8.5) circle (6pt);
\filldraw[black,fill=white] (0.5,7.5) circle (6pt);
\end{tikzpicture}
\begin{tikzpicture}[scale=0.25]
\dyck{0,0}{9}{black}{0,0,0,1,1,0,0,1,0,0,1,0,1,1,1,1,0,1}
\draw[dashed] (1,0) -- (8,7); 
\filldraw[black] (2.5,0.5) circle (6pt);
\filldraw[black] (4.5,2.5) circle (6pt);
\filldraw[black] (6.5,3.5) circle (6pt);
\filldraw[black] (7.5,4.5) circle (6pt);
\filldraw[black] (8.5,8.5) circle (6pt);
\filldraw[black] (0.5,7.5) circle (6pt);
\filldraw[black,fill=white] (1.5,1.5) circle (6pt);
\filldraw[black,fill=white] (3.5,6.5) circle (6pt);
\end{tikzpicture}
\begin{tikzpicture}[scale=0.25]
\dyck{0,0}{9}{black}{0,0,0,1,1,0,0,1,0,0,1,0,1,1,1,1,0,1}
\draw[dashed] (5,3) -- (8,6); 
\filldraw[black] (2.5,0.5) circle (6pt);
\filldraw[black] (4.5,2.5) circle (6pt);
\filldraw[black] (6.5,3.5) circle (6pt);
\filldraw[black] (7.5,4.5) circle (6pt);
\filldraw[black] (8.5,8.5) circle (6pt);
\filldraw[black] (0.5,7.5) circle (6pt);
\filldraw[black] (1.5,1.5) circle (6pt);
\filldraw[black] (3.5,6.5) circle (6pt);
\filldraw[black,fill=white] (5.5,5.5) circle (6pt);
\end{tikzpicture}
\begin{tikzpicture}[scale=0.25]
\fill[white]  (0,0) rectangle +(9,9);
\draw[help lines] (0,0) grid +(9,9);
\filldraw[black] (2.5,0.5) circle (6pt);
\filldraw[black] (4.5,2.5) circle (6pt);
\filldraw[black] (6.5,3.5) circle (6pt);
\filldraw[black] (7.5,4.5) circle (6pt);
\filldraw[black] (8.5,8.5) circle (6pt);
\filldraw[black] (0.5,7.5) circle (6pt);
\filldraw[black] (1.5,1.5) circle (6pt);
\filldraw[black] (3.5,6.5) circle (6pt);
\filldraw[black] (5.5,5.5) circle (6pt);
\end{tikzpicture}
\caption{\small{An example of a reversible process (left to right) of associating a unique 231-avoider with a given Dyck path.~\cite{albert-new}.}}
\label{fig:bij231dyck}
\end{center}
\end{figure}
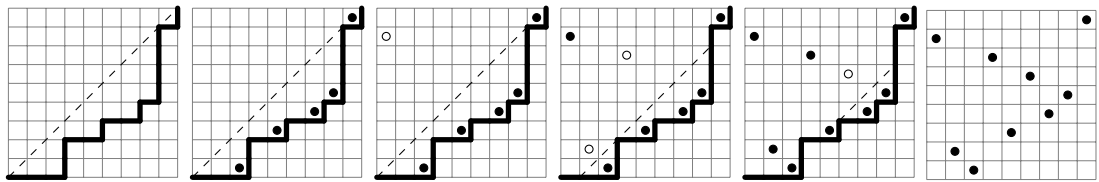

This paragraph is best read alongside Figure~\ref{fig:bij231dyck}. To obtain a 231-avoider of length $n$ from a Dyck path $\DD$, one places a point into each \emph{corner} (a down step followed by a left step) to obtain the right-to-left minima of the permutation to be constructed. This is the second diagram from the left in Figure~\ref{fig:bij231dyck}.  We place the remaining points of the permutation for each excursion of $\DD$ separately. Given such an excursion, the first and last steps are not consecutive/adjacent and must be a down step and a left step, respectively. Insert a point in the square where the respective row and column of these steps meet. This is documented in the third diagram of Figure~\ref{fig:bij231dyck}. The points inserted in a given step are marked by a circle. Do the same for each excursion of $\DD$. In the next step offset the diagonal, disregard the points above the new diagonal, and repeat the process. It is easy to check that this procedure is correct (gives a 231-avoider) and reversible, since the corners of the Dyck path are right-to-left minima.

\begin{figure}[!ht]
\begin{center}
\begin{tikzpicture}[scale=0.2]
\dyck{0,0}{9}{black}{0,0,0,1,1,0,0,1,0,0,1,0,1,1,1,1,0,1}
\end{tikzpicture}
\begin{tikzpicture}[scale=0.2]
\dyck{0,0}{9}{black}{0,0,0,1,1,0,0,1,0,0,1,0,1,1,1,1,0,1}
\filldraw[black] (2.5,0.5) circle (6pt);
\filldraw[black] (4.5,2.5) circle (6pt);
\filldraw[black] (6.5,3.5) circle (6pt);
\filldraw[black] (7.5,4.5) circle (6pt);
\filldraw[black] (8.5,8.5) circle (6pt);
\end{tikzpicture}
\begin{tikzpicture}[scale=0.2]
\dyck{0,0}{9}{black}{0,0,0,1,1,0,0,1,0,0,1,0,1,1,1,1,0,1}
\filldraw[black] (2.5,0.5) circle (6pt);
\filldraw[black] (4.5,2.5) circle (6pt);
\filldraw[black] (6.5,3.5) circle (6pt);
\filldraw[black] (7.5,4.5) circle (6pt);
\filldraw[black] (8.5,8.5) circle (6pt);
\filldraw[black,fill=white] (5.5,7.5) circle (6pt);
\end{tikzpicture}
\begin{tikzpicture}[scale=0.2]
\dyck{0,0}{9}{black}{0,0,0,1,1,0,0,1,0,0,1,0,1,1,1,1,0,1}
\filldraw[black] (2.5,0.5) circle (6pt);
\filldraw[black] (4.5,2.5) circle (6pt);
\filldraw[black] (6.5,3.5) circle (6pt);
\filldraw[black] (7.5,4.5) circle (6pt);
\filldraw[black] (8.5,8.5) circle (6pt);
\filldraw[black] (5.5,7.5) circle (6pt);
\filldraw[black,fill=white] (3.5,6.5) circle (6pt);
\end{tikzpicture}
\begin{tikzpicture}[scale=0.2]
\dyck{0,0}{9}{black}{0,0,0,1,1,0,0,1,0,0,1,0,1,1,1,1,0,1}
\filldraw[black] (2.5,0.5) circle (6pt);
\filldraw[black] (4.5,2.5) circle (6pt);
\filldraw[black] (6.5,3.5) circle (6pt);
\filldraw[black] (7.5,4.5) circle (6pt);
\filldraw[black] (8.5,8.5) circle (6pt);
\filldraw[black] (5.5,7.5) circle (6pt);
\filldraw[black] (3.5,6.5) circle (6pt);
\filldraw[black,fill=white] (1.5,5.5) circle (6pt);
\end{tikzpicture}
\begin{tikzpicture}[scale=0.2]
\dyck{0,0}{9}{black}{0,0,0,1,1,0,0,1,0,0,1,0,1,1,1,1,0,1}
\filldraw[black] (2.5,0.5) circle (6pt);
\filldraw[black] (4.5,2.5) circle (6pt);
\filldraw[black] (6.5,3.5) circle (6pt);
\filldraw[black] (7.5,4.5) circle (6pt);
\filldraw[black] (8.5,8.5) circle (6pt);
\filldraw[black] (5.5,7.5) circle (6pt);
\filldraw[black] (3.5,6.5) circle (6pt);
\filldraw[black] (1.5,5.5) circle (6pt);
\filldraw[black,fill=white] (0.5,1.5) circle (6pt);
\end{tikzpicture}
\begin{tikzpicture}[scale=0.2]
\fill[white]  (0,0) rectangle +(9,9);
\draw[help lines] (0,0) grid +(9,9);
\filldraw[black] (2.5,0.5) circle (6pt);
\filldraw[black] (4.5,2.5) circle (6pt);
\filldraw[black] (6.5,3.5) circle (6pt);
\filldraw[black] (7.5,4.5) circle (6pt);
\filldraw[black] (8.5,8.5) circle (6pt);
\filldraw[black] (5.5,7.5) circle (6pt);
\filldraw[black] (3.5,6.5) circle (6pt);
\filldraw[black] (1.5,5.5) circle (6pt);
\filldraw[black] (0.5,1.5) circle (6pt);
\end{tikzpicture}
\caption{\small{An example of a reversible process (left to right) of associating a unique 321-avoider with a given Dyck path.}}
\label{fig:bij321dyck}
\end{center}
\end{figure}
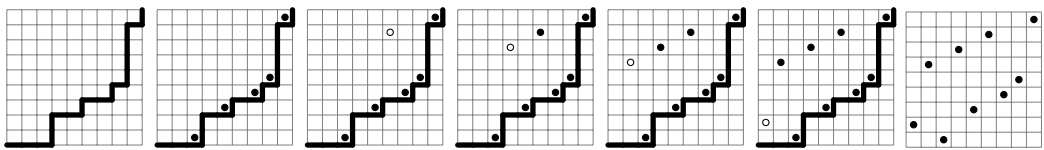

This paragraph is accompanied by Figure~\ref{fig:bij321dyck}. To obtain a 321-avoider of length $n$ from a Dyck path $\DD$, place right-to-left minima as in the case above (231-avoider to Dyck path). Then, find the first down step that is not immediately succeeded by a left step and find the first left step that is not immediately preceded by a down step. Place a point in the square where the row and column of these two meet. This step is shown in the third diagram in Figure~\ref{fig:bij321dyck} --- the new point at each step is denoted by a circle. Add the rest of the points for the further down steps and left steps in the same fashion.

Having fixed bijections between Dyck paths and 231-avoiders and Dyck paths and 321-avoiders, we will refer to the Dyck path \emph{corresponding} to a permutation $P$, and vice versa, whenever one of the bijections transforms one into the other.

A \emph{context-free grammar (CFG)} is a formal grammar that describes a language consisting of only those words which can be obtained from a starting string by repeated use of permitted production rules/substitutions. Formally, a context-free grammar is a four-tuple $(V, \Sigma, R, \fS )$: a finite set of \emph{variables} $V$, a set of terminal characters (symbols) $\Sigma$, a finite relation $R$ from $V$ to $(V \cup \Sigma)^*$ where $*$ is the Kleene star, and the start variable $\fS$, $\fS \in V$. The pairs in $R$ are called \emph{production rules}, because they specify what we are allowed to substitute into a given variable. In our context-free grammars, the starting symbol is always $\fS$ (unless CFG describes a Catalan object, then it is $\fC$) and $\epsilon$ denotes an empty word. A good reference for formal languages is~\cite{hopcroft2001automata}.

\begin{example}[Enumerating $\Av(231)$]
Enumerating this class serves as a template for how we intend to use CFGs in Section~\ref{sec:enum}. Recall that Dyck paths are in one-to-one correspondence with 231-avoiding permutations. We use this fact to enumerate permutations in $\Av(231)$ by counting Dyck paths. A Dyck path corresponding to a 231-avoider (we start in the bottom left corner, and stay below the diagonal) is either empty or starts with a step right, followed by an arbitrary Dyck path, followed by a step up, followed by an arbitrary Dyck path. Denote by $\fC$ a Dyck path ($\fC$ stands for Catalan object), by $\epsilon$ the empty Dyck path, and by $\fU$, $\fR$ up step and right step respectively. Then the CFG for Dyck paths (and hence also for 231- and 321-avoiders) looks as follows
\begin{align*}
\fC &\to \epsilon \mid \fR\fC\fU\fC.
\end{align*}
\noindent As written, this grammar is unambiguous, and can be transcribed syntactically into a functional equation. We let $z$ count the number of $up$ steps (or alternatively the number of right steps) in the Dyck path. We find that $c = c(z)$ satisfies the relation below 
\begin{align*}
c &= 1 + zc^2.
\end{align*}
\noindent Consequently, we obtain $c$ as a function of $z$ by ``solving for'' $c$ formally
\begin{align*}
c(z) &= \frac{1 +\sqrt{1-4z}}{2z}.
\end{align*}
\end{example}

\section{Enumeration}
\label{sec:enum}

Let $P \in \Av(abc)$ and let $\P$ be the corresponding Dyck path, and let $X \in \Av(xy)$. Recall that when enumerating elements of $\Av(abc|xy)$, we chose to place the vertical gridline as far right as possible.  If it was any further right, there would be a point on the RHS which would serve as $c$ in a copy of $abc$. In Figure~\ref{fig:gridded}, for instance, there is no copy of 231 on the LHS, but if the vertical gridline was shifted one place to the right, the three unfilled points would form a copy of 231 on the LHS of the gridline. The gridline is already as far right as possible. The permutations gridded this way are exactly the $\Av(abc|xy)$-griddable permutations and we are free to enumerate these gridded permutations. The setup will be as in Figure~\ref{fig:translation}. A Dyck path $\P$ on the left decorated by sequences of points from the right. We will always specify whether a sequence of points corresponding to a vertical step $V$ in $\P$ is placed immediatelly below $V$ or immediatelly above $V$. Figure~\ref{fig:decorated} shows the ``below''case. Every griddable permutation from the juxtaposition class has a unique gridding that maximizes the size of the $\Av(abc)$ class, i.e.~pushing the gridline as far to the right as possible. We call the leftmost point on the RHS the \emph{first} point on the RHS. We say that a point $r$ on the RHS is \emph{enclosed} by two points $p,q$ on the LHS if $p$ is above $r$ and $r$ is above $q$ in the drawing of the juxtaposition. If $\DD'$ is a section of a Dyck path $\DD$ on the LHS, we say that a point $r$ on the RHS is \emph{contained} in $\DD'$ whenever the end points of $\DD'$ enclose $r$. A \emph{Catalan block} $\DD'$ (or a \emph{Dyck block}) is a section of a Dyck path $\DD$ which is a Dyck path with respect to some offset of the diagonal (it begins and ends on it and stays on one side of it at all times). Notice that a Dyck block is essentially a collection of consecutive excursions that all have the same offset from the diagonal.

\begin{figure}[h!]
\centering
\subcaptionbox{Sequences of points are placed at the bottom of the corresponding vertical steps (highlighted).\label{fig:decorated}}[0.45\textwidth]{
\begin{tikzpicture}[scale=0.3]
\dyck{0,0}{14}{black}{0,1,0,0,1,0,0,1,1,1,0,1,0,0,0,1,1,0,1,0,0,1,1,0,0,1,1,1}
\filldraw[black] (13.5,11.5) circle (6pt);
\filldraw[black] (11.5,9.5) circle (6pt);
\filldraw[black] (9.5,8.5) circle (6pt);
\filldraw[black] (8.5,6.5) circle (6pt);
\filldraw[black] (5.5,5.5) circle (6pt);
\filldraw[black] (4.5,2.5) circle (6pt);
\filldraw[black] (2.5,1.5) circle (6pt);
\filldraw[black] (.5,.5) circle (6pt);
\filldraw[black] (3.5,3.5) circle (6pt);
\filldraw[black] (1.5,4.5) circle (6pt);
\filldraw[black] (6.5,13.5) circle (6pt);
\filldraw[black] (7.5,7.5) circle (6pt);
\filldraw[black] (10.5,10.5) circle (6pt);
\filldraw[black] (12.5,12.5) circle (6pt);
\draw[gray] (14,10) edge[out=0, in=-160] (18,11);
\draw[gray] (14,10) edge[out=0, in=160] (18,9);
\draw[gray] (14,7) edge[out=0, in=-160] (20,8);
\draw[gray] (14,7) edge[out=0, in=160] (20,6);
\draw[gray] (14,3) edge[out=0, in=-160] (22,4);
\draw[gray] (14,3) edge[out=0, in=160] (22,2);
\filldraw[black] (18.25,10.5) circle (3pt);
\filldraw[black] (18.5,10.25) circle (3pt);
\filldraw[black] (18.75,10) circle (3pt);
\filldraw[black] (19,9.75) circle (3pt);

\filldraw[black] (20.25,7.25) circle (3pt);
\filldraw[black] (20.5,7) circle (3pt);
\filldraw[black] (20.75,6.75) circle (3pt);

\filldraw[black] (22.25,3.25) circle (3pt);
\filldraw[black] (22.5,3) circle (3pt);
\filldraw[black] (22.75,2.75) circle (3pt);
\filldraw[black] (23,2.5) circle (3pt);
\filldraw[black] (23.25,2.25) circle (3pt);
\filldraw[black] (23.5,2) circle (3pt);
\draw[Orange, line width=0.7mm] (12,10) -- (12,11);
\draw[Orange, line width=0.7mm] (9,7) -- (9,8);
\draw[Orange, line width=0.7mm] (5,3) -- (5,4);

\filldraw[white] (0,-2) circle (3pt);

\end{tikzpicture}}\hfill
\subcaptionbox{Gridline is pushed as far right as possible: the distinguished points would form a copy of 231 on the LHS.\label{fig:gridded}}[0.45\textwidth]{
\begin{tikzpicture}[scale=0.2]
\fill[white]  (0,0) rectangle +(27,27);
\draw[help lines] (0,0) grid +(27,27);
\filldraw[black,fill=white] (13.5,24.5) circle (6pt);
\filldraw[black] (12.5,25.5) circle (6pt);
\filldraw[black] (11.5,18.5) circle (6pt);
\filldraw[black,fill=white] (10.5,23.5) circle (6pt);
\filldraw[black] (9.5,17.5) circle (6pt);
\filldraw[black] (8.5,12.5) circle (6pt);
\filldraw[black] (7.5,16.5) circle (6pt);
\filldraw[black] (6.5,26.5) circle (6pt);
\filldraw[black] (5.5,11.5) circle (6pt);
\filldraw[black] (4.5,2.5) circle (6pt);
\filldraw[black] (3.5,9.5) circle (6pt);
\filldraw[black] (2.5,1.5) circle (6pt);
\filldraw[black] (1.5,10.5) circle (6pt);
\filldraw[black] (.5,.5) circle (6pt);
\filldraw[black,fill=white] (14.5,22.5) circle (6pt);
\filldraw[black] (15.5,21.5) circle (6pt);
\filldraw[black] (16.5,20.5) circle (6pt);
\filldraw[black] (17.5,19.5) circle (6pt);

\filldraw[black] (18.5,15.5) circle (6pt);
\filldraw[black] (19.5,14.5) circle (6pt);
\filldraw[black] (20.5,13.5) circle (6pt);

\filldraw[black] (21.5,8.5) circle (6pt);
\filldraw[black] (22.5,7.5) circle (6pt);
\filldraw[black] (23.5,6.5) circle (6pt);
\filldraw[black] (24.5,5.5) circle (6pt);
\filldraw[black] (25.5,4.5) circle (6pt);
\filldraw[black] (26.5,3.5) circle (6pt);
\draw[black, line width=0.5mm] (14,-0.5) -- (14,27.5);
\end{tikzpicture}}
\caption{\small On the left is a decorated Dyck path of a 231-avoider, while on the right is the corresponding gridded permutation from $\Av(231|12)$. Every griddable permutation in $\Av(231|12)$ has exactly one such gridding.}
\label{fig:translation}
\end{figure}

\subsection{Class $\mathbf{A} = \Av(231|12)$}
\label{subsec:davidsclass}
As we mentioned already, a symmetry of this class was enumerated by Bevan~\cite{bevan-new} as a step towards the enumeration of $\Av(4213,2143)$, by exploiting a tree-like structure of the permutations in $\Av(231)$. Here, we enumerate $\Av(231|12)$ by first describing through a context-free grammar.

We represent a 231-avoiding permutation $P$ by a Dyck path $\P$ from top-right to bottom-left, and below the diagonal. Sequences of points (each possibly empty) on the RHS are placed immediately below the corresponding down steps. Again, the gridding of $P \in \Av(231|12)$ is chosen to maximize the number of points on the LHS. Therefore, the leftmost point on the RHS must be below the ``1'' in the topmost 12 on the LHS. In a corresponding Dyck path $\P$, this means that we can start placing point sequences on the RHS as soon as we have encountered a down step ($\fD$) in $\P$ that has a left step ($\fL$) before it. The following letters will be needed to build $\P$. 
\begin{enumerate}
\item[$\fL$ --] left step
\item[$\fD$ --] down step before any left steps occured
\item[$\bD$ --] down step after left step already occured
\end{enumerate}

We denote by $\bC$ a Dyck path over letters $\fL$ and $\bD$, while $\fC$ is a standard Dyck path over $\fL$ and $\fD$. Given these building blocks, a Dyck path representing a 231-avoider can only start with a $\fD$, followed by a possibly empty Dyck path, and returning to the diagonal with a left-step $\fL$. Afterwards any other Dyck path can be appended as long as it is over letters $\fL, \bD$. Hence, the CFG describing the elements of $\Av(231|12)$ has the following rules.

\begin{align*}
\fS &\to \epsilon \mid \fD\fS\fL\bC\\
\bC &\to \epsilon \mid \bD\bC\fL\bC
\end{align*}

From here we can directly pass to the functional equations by letting $z$ track the down steps (whether $\fD$ or $\bD$) and $t$ track the sequences of points in the right compartment of the grid. In particular, $\bD$ transcribes as $tz$. Both $s$ and $\bc$ are functions of $t$ and $z$.

\begin{align*}
s &= 1 + zs\bc\\
\bc &= 1 + tz\bc^2
\end{align*}

Since there is no ambiguity to how we place points in the right compartment of the grid, only decreasing sequences are allowed, we let $t = 1/(1-z)$. Solving for $s$ gives the next theorem. 

\begin{theorem}[Bevan~\cite{bevan-new}]
The generating function for $\mathbf{A} = \Av(231|12)$ is 
\begin{align*}
s_{\mathbf{A}}(z)
&= \frac{1+z-\sqrt{1-6z+5z^2}}{2z(2-z)}
\end{align*}
\end{theorem}

The generating function $s_{\mathbf{A}}(z)$ stores terms of \href{http://oeis.org/A033321}{A033321} in OEIS~\cite{oeis}. The initial twelve coefficients of $s_{\mathbf{A}}(z)$ are 1, 2, 6, 21, 79, 311, 1265, 5275, 22431, 96900, 424068, 1876143. Again, see Bevan~\cite{bevan-new} for a different approach to enumerating this juxtaposition class. Also, Miner~\cite{miner16twobyfour} recently enumerated $\Av(321|12)$ (via yet another approach), which is in bijection with $\Av(231|12)$ as we show in Section~\ref{sec:bijections}.

\subsection{Class $\mathbf{B} = \Av(321|21)$}
We represent a 321-avoiding permutation $P$ by a Dyck path $\P$ from the bottom left to the top right corner of the grid, and staying above the diagonal. The sequences of points (each possibly empty) on the RHS, associated with vertical steps (up steps) on the LHS, are always placed directly below the corresponding $\bU$ on the LHS (i.e. below the point in $P$ associated with the $\bU$). Additionally, we need to correct for the points on the RHS which can occur above everything on the LHS. We do this when transcribing the grammar to equations. Below is the alphabet used for our CFG.

\begin{enumerate}
\item[$\fU$ --] up step, no sequence of points on the RHS associated with it
\item[$\fR$ --] right step
\item[$\bU$ --] up step with a possible sequence of points right below it on RHS
\item[$\bU_1$ --] up step associated with the first point on the RHS 
\end{enumerate}

\noindent The CFG rules are below. Starting state is $\fS$.

\begin{align*}
\fS &\to \fC \mid \fC\fM\bC\\
\fM &\to \fU\fC\fM\bC\fR \mid \bU_1\bC^+\fR \mid \fB\bE\bU\bC^+\fR\\
\fB &\to \bU_1\fR \mid \fU\fC\fB\bE\fR\\
\fC &\to \epsilon \mid \fU\fC\fR\fC\\
\bC &\to \epsilon \mid \bU\bC\fR\bC\\
\bC^+ &\to \bU\bC\fR\bC\\
\bE &\to \epsilon \mid \bU\fR\bE
\end{align*}

\noindent We now describe the rules above in detail. 

\begin{enumerate}
\item[$\fS$ --] Start rule. The entire object is either a Catalan object without points on the RHS, or there is at least one point on the RHS. In the latter case, the entire word starts with a Catalan object $\fC$, then there is the middle part $\fM$ containing $\bU_1$ and a 21 above this $\bU_1$. All this is followed by a possibly empty Catalan object $\bC$ that allows points on the RHS.
\item[$\fM$ --] Middle part (possibly multiple excursions) with $\bU_1$ and the first 21 above this $\bU_1$. Part $\fM$ either starts with a point on the RHS and then $\bU_1\bC^+\fR$ guarantees a 21 in the same Catalan block that sits on the diagonal (the $\bU_1$ is not in the corner). Or we first see the initial point on the RHS without a 21 above it and in the same Catalan block (call this block $B$) but we then see a 21 later (in this case, the $\bU_1$ is in the corner, i.e. not in $\bU_1\fR$). This is the case $\fB\bE\bU\bC^+\fR$. Or we offset the diagonal and repeat the process, i.e.~we recur $\fM$. That is, $\fU\fC\fM\bC\fR$.
\item[$\fB$ --] An excursion with the first point on the RHS and no 21 above it in the same excursion. Either we are in the base case, $\bU_1\fR$, or we recur the construction of $\fB$. That is, nothing before $\fB$ can contain a point on the RHS, and nothing after $\fB$ can contain a 21. We get $\fU\fC\fB\bE\fR$ as desired.
\item[$\fC$ --] Catalan block without restrictions.
\item[$\bC$ --] Catalan block with points allowed on the RHS.
\item[$\bE$ --] Stairs with points allowed on the RHS (repeated $\bU\fR$).
\item[$\bC^+$ --] Catalan block that is not empty.
\end{enumerate}

\noindent We obtain the following equations from the rules above. The additional $t$ at the end of the first equation tracks the gap on the RHS above everything on the LHS. This gap could contain a non-empty increasing sequence of points. 

\begin{align*}
s &= c + cm\bc t\\
b &= z^2t + zcb\be\\
m &= zcm\bc + b\be zt(\bc -1) + z^2t(\bc -1)\\
c &= 1+zc^2\\
\bc &= 1+ zt\bc^2\\
\be &= 1+zt\be
\end{align*}
Given the role of $t$, to track sequences of points whose position is determined, we can again replace it with $1/(1-z)$. After the substitution and after solving the above equations, we find the following theorem.
\begin{theorem}
The generating function for $\mathbf{B} = \Av(321|21)$ is 
\begin{align*}
  s_{\mathbf{B}}(z) &= -\frac{1-\sqrt{1-4z} + z(-4+\sqrt{1-4z} + \sqrt{1-5z}/\sqrt{1-z})}{2z^2}
\end{align*}
\end{theorem}

The generating function $s_{\mathbf{B}}(z)$ stores the sequence \href{http://oeis.org/A278301}{A278301} of OEIS~\cite{oeis}. The initial twelve coefficients of $s_{\mathbf{B}}(z)$ are 1, 2, 6, 23, 98, 434, 1949, 8803, 39888, 181201, 825201, 3767757.

\subsection{Class $\mathbf{C} = \Av(312|21)$}
\label{sec:classC}
We represent a 312-avoiding permutation $P$ by a Dyck path $\P$ starting in the bottom-left and ending in the top-right corner of the grid, and staying above the diagonal. The sequences of points (each possibly empty) on the RHS associated with $\bU$s are placed immediately above those respective $\bU$s. The letters in the alphabet stay the same as above when we described the class $\mathbf{B}$.

\begin{enumerate}
\item[$\fU$ --] up step, no sequence of points on the RHS associated with it
\item[$\fR$ --] right step
\item[$\bU$ --] up step with a possible sequence of points right above it on RHS
\item[$\bU_1$ --] first $\bU$, marks the up step with the first point on the RHS
\end{enumerate}

\noindent The following rules describe the context-free grammar enumerating $\Av(312|21)$.  

\begin{align*}
\fS &\to \fC \mid \fC\fM\bC\\
\fM &\to \bU_1\bC^+\fR \mid \fU\fC\fB\bC^+\fR \mid \fU\fC\fM\bC\fR\\
\fB &\to \bU_1\fR \mid \fU\fC\fB\fR\\
\fC &\to \epsilon \mid \fU\fC\fR\fC\\
\bC &\to \epsilon \mid \bU\bC\fR\bC\\
\bC^+ &\to \bU\bC\fR\bC
\end{align*}

\noindent We now justify the above rules.

\begin{enumerate}
\item[$\fS$ --] the entire sequence $\fS$ is either a Catalan object $\fC$ (no points on the RHS) or it has a distinguished middle part $\fM$ which is a Catalan object sitting on the diagonal which contains the first point on the RHS. Given the way a 312-avoiding permutation sits in the Dyck path, it follows that $\fM$ also contains the 21 that encloses this first point on the RHS. Such $\fM$ is then followed by $\bC$ allowing points on the RHS.
\item[$\fM$ --] The middle part is either one of the two base cases, or the recursive case. In the base cases, the first point on the RHS occurs in $\fM$ -- either in a corner ($\bU_1\fR$) or not. If not in the corner, then the rule to capture this is $\bU_1\bC^+\fR$ -- the Catalan object must be non-empty as this part sits on the diagonal and so the point on the RHS associated with $\bU_1$ must be enclosed by a 21 on the LHS. If the first point on the RHS occurs in the corner, then we refer to $\fB$, i.e.~the rule is $\fU\fC\fB\bC^+\fR$. Notice that we want the Catalan object after $\fB$ to be non-empty, as $\fB$ only contains $\fR$ steps after a $\bU_1$. Finally, the recursive step is $\fU\fC\fM\bC\fR$ as expected.
\item[$\fB$ --] this part describes the Catalan object containing $\bU_1$ (first up-step which is immediately followed by a point on the RHS) that is in the corner -- i.e.~forms $\bU_1\fR$. Notice that $\fB$ cannot sit on the diagonal, for otherwise the first point on the RHS would not be enclosed in a 21 on the LHS. So the base case for $\fB$ is just a corner with a point on the RHS, and the recursive step is an object $\fB$ preceded by an arbitrary Catalan object, this all offset from the start level. The way $\fB$ is used inside $\fM$ makes it unnecessary to have $\fU\fC\fB\bC\fR$ instead of $\fU\fC\fB\fR$. It would lead to double-counting.
\item[$\fC$ --] Catalan block without any restrictions.
\item[$\bC$ --] Catalan block with arbitrary point sequences allowed on the RHS immediately after every up step on the LHS. 
\item[$\bC^+$ --] Catalan block that is not empty.
\end{enumerate} 

\noindent The equations that follow from the above rules are as follows (in the same order). The functions $s, m, b, \bc, c$ take arguments $z$ and $t$. As before, set $t = 1/(1-z)$ as it represents a (possibly empty) sequence of points on the RHS.
\begin{align*}
s &= c + cm\bc\\
m &= z^2t(\bc -1) + zcb(\bc -1) + zcm\bc\\
b &= zcb + z^2t\\
c &= 1 + c^2z\\
\bc &= 1 + \bc^2zt
\end{align*}

\begin{theorem}
The generating function for the class $\mathbf{C} = \Av(312|21)$ is
\begin{align*}
s_{\mathbf{C}}(z) &= -\frac{(1+\sqrt{1-4z})(1-z+\sqrt{x-1}\sqrt{5x-1})}{4z}
\end{align*}
\end{theorem}

The initial twelve coefficients of $s_{\mathbf{C}}(z)$ are 1, 2, 6, 22, 88, 367, 1568, 6810, 29943, 132958, 595227, 2683373. It turns out that this sequence enumerating $\Av(312|21)$ is the sequence \href{http://oeis.org/A165538}{A165538} in \href{http://oeis.org/}{OEIS}. As we mention in Section~\ref{sec:defs}, \href{http://oeis.org/A165538}{A165538} also enumerates the class $\Av(4312,3142)$, see~\cite{albert2012inflations}.

\section{Bijections}
\label{sec:bijections}
Before we describe bijections between the juxtaposition classes, let us introduce the notion of an \emph{articulation point} in 231-avoiders and 321-avoiders in order to define a canonical bijection between the two classes (in which the articulation point is fixed).

\begin{figure}[h!]
\centering
\begin{subfigure}{0.45\textwidth}
\begin{center}
\begin{tikzpicture}[scale=0.3]
\dyck{0,0}{14}{black}{0,1,0,0,1,0,0,1,1,1,0,1,0,0,0,1,1,0,1,0,0,1}
\pdyck{12,10}{Orange}{1,0,0,1,1,1}
\draw[black] (12,10) circle (10pt);
\filldraw[black] (13.5,11.5) circle (6pt);
\filldraw[black] (11.5,9.5) circle (6pt);
\filldraw[black] (9.5,8.5) circle (6pt);
\filldraw[black] (8.5,6.5) circle (6pt);
\filldraw[black] (5.5,5.5) circle (6pt);
\filldraw[black] (4.5,2.5) circle (6pt);
\filldraw[black] (2.5,1.5) circle (6pt);
\filldraw[black] (.5,.5) circle (6pt);
\filldraw[black] (3.5,3.5) circle (6pt);
\filldraw[black] (1.5,4.5) circle (6pt);
\filldraw[black] (6.5,13.5) circle (6pt);
\filldraw[black] (7.5,7.5) circle (6pt);
\filldraw[black] (10.5,10.5) circle (6pt);
\filldraw[black] (12.5,12.5) circle (6pt);
\end{tikzpicture}
\end{center}
\caption{\small 231-avoider.}
\label{fig:articulation231}
\end{subfigure}\begin{subfigure}{0.45\textwidth}
\begin{center}
\begin{tikzpicture}[scale=0.3]
\dyck{0,0}{14}{black}{0,1,0,0,1,0,0,1,1,1,0,1,0,0,0,1,1,0,1,0,0,1}
\pdyck{12,10}{Orange}{0,1,1,1,0,1}
\draw[black] (12,10) circle (10pt);
\filldraw[black] (13.5,13.5) circle (6pt);
\filldraw[black] (12.5,10.5) circle (6pt);
\filldraw[black] (11.5,9.5) circle (6pt);
\filldraw[black] (9.5,8.5) circle (6pt);
\filldraw[black] (8.5,6.5) circle (6pt);
\filldraw[black] (5.5,5.5) circle (6pt);
\filldraw[black] (4.5,2.5) circle (6pt);
\filldraw[black] (2.5,1.5) circle (6pt);
\filldraw[black] (.5,.5) circle (6pt);
\filldraw[black] (1.5,3.5) circle (6pt);
\filldraw[black] (3.5,4.5) circle (6pt);
\filldraw[black] (6.5,7.5) circle (6pt);
\filldraw[black] (7.5,11.5) circle (6pt);
\filldraw[black] (10.5,12.5) circle (6pt);
\end{tikzpicture}
\end{center}
\caption{\small 321-avoider.}
\label{fig:articulation321}
\end{subfigure}
\caption{\small Articulation points in a 231-avoider and 321-avoider (left and right, respectively) at the position $(12,10)$, marked by empty circles in both pictures.}
\label{fig:theta}
\end{figure}
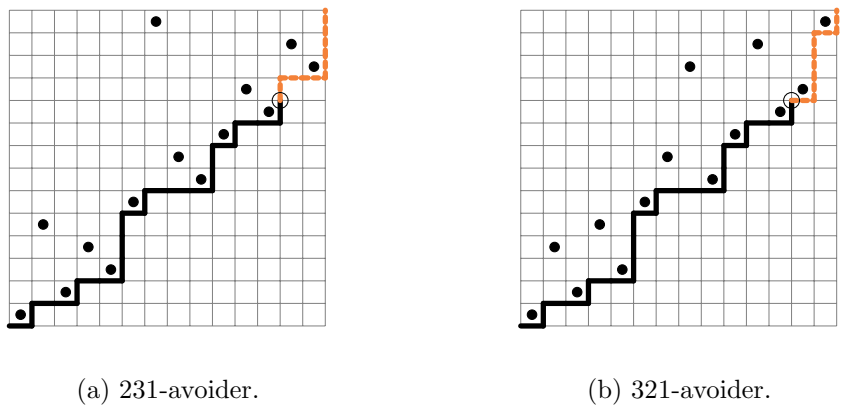

Let $P$ be a $231$-avoider and $\P$ the corresponding Dyck path below the diagonal from the top-right corner to the bottom-left corner. The articulation point $x$ on $\P$ is any and all of the following.
\begin{itemize}
\item The end-point of the first down step after a left-step has occured in $\P$.
\item The intersection of $\P$ and the gridline directly below the ``1'' in the topmost copy of 12 in $P$.
\end{itemize}
The two descriptions are readily seen equivalent from Figure~\ref{fig:articulation231} where $P$ and $\P$ are drawn into the same picture. It follows that an articulation point $x$ partitions $\P$ into $\P'$ (from the top to $x$) and $\P''$ (from $x$ to the bottom). Notice that $\P'$ is unique in the sense that it is the very part of a Dyck path that for which $x$ is an articulation point. In other words, there is one and only one $\P'$ that completes $\P''$ into $\P$. To be exact, $\P'$ must end with a $\fD$ (down step) and this $\fD$ must be the first $\fD$ after an $\fL$ occured. Hence, $\P' = \fX\fD$, where $\fX$ is a string of consecutive $\fD$s followed by consecutive $\fL$s.

Let $Q$ be a $321$-avoider and $\Q$ its Dyck path below the diagonal from the top-right corner to the bottom-left corner. Then the articulation point $y$ on $\Q$ is defined as any/all of the following.
\begin{itemize}
\item It is the end-point of the first left step that does not end on the diagonal.
\item It is the bottom left corner of the grid box of ``1'' in the topmost 21 in $Q$ which is simultaneously also a point on $\Q$.
\end{itemize}
Again, example is in Figure~\ref{fig:articulation321}. Therefore, $y$ partitions $\Q$ into $\Q'$ (from the top to $y$) and $\Q''$ (from $y$ to the bottom). Given $\Q''$, there is a unique $\Q'$ that completes it into $\Q$. Indeed, $\Q'$ must end with an $\fL$. It must start with a sequence of $\fD\fL$ pairs, followed by a sequence of $\fD$s until the final letter $\fL$.

By an articulation point of a permutation we mean the articulation point of the corresponding Dyck path. Next lemma tells us that we can pass between the two classes, $\Av(231)$ and $\Av(321)$, and keep the articulation point fixed.

\begin{lemma}
\label{lem:simplebijection}
There exists a bijective map $\lambda: \Av(231) \to \Av(321)$ that fixes articulation points.
\end{lemma}
\begin{proof}
Given the position of an articulation point $x$, let $\P''$ and $\Q''$ be identical. Then there are unique completions $\P'$ and $\Q'$, of $\P''$ and $\Q''$, that make them into Dyck paths $\P$ and $Q$, respectively. Clearly, this is a bijection, say $\lambda'$, between Dyck paths $\P$ with an articulation point $x$ and Dyck paths $\Q$ with the same articulation point $x$. Given the bijections between 231-avoiders and Dyck paths, and 321-avoiders and Dyck paths from Section~\ref{sec:defs}, we can define $\lambda$ as a composition of these with $\lambda'$. This makes $\lambda$ a bijection as desired.
\begin{align*}
\Av(231) \to \text{Dyck}(\Av(231)) \overset{\lambda'}{\to} \text{Dyck}(\Av(321))\to \Av(321).
\end{align*}
\end{proof}

For later use, we define two more notions. Firstly, when we write $P = P_1P_2$ we mean that $P_1$ and $P_2$ to be the permutations on the LHS and the RHS of the gridding of $P$ that maximizes the size of $P_1$, i.e.~if $P_1$ contained the first point of $P_2$, it would not be an $abc$-avoider. In particular, note the interplay of an articulation point and the vertical gridline. If the gridline was more to the left, then the (possibly different) articulation point would be below the first point on the RHS. Secondly, assume the down steps (vertical steps) in a Dyck path $\P$ corresponding to a permutation $P_1$ are labelled 1 through $n$ in order of their appearance in $\P$. Then we can describe the permutation $P_2$ on the RHS of the gridding by a sequence $h$ of length $n$. Given $h = (h_1,\ldots,h_n)$, $h_i$ is the length of the sequence of points associated with the $i$-th down step $D_i$. Depending on our choice (see Section~\ref{sec:enum}), these sequences of $h_i$ points are placed either immediately above or immediately below the point of the permutation corresponding to $D_i$. In any case, such a sequence $h$ together with a Dyck path $\P$ uniquely describe a permutation from one of the juxtaposition classes $\Av(abc|xy)$. Therefore, we can refer to a $P = P_1P_2 \in \Av(231|21)$ as a pair $(P,h_P)$ (or alternatively $(P_1,h_P)$) for the appropriate sequence $h_P$.

Theorem~\ref{thm:theta} and~\ref{thm:psi} both make use of the bijection $\lambda$ from Lemma~\ref{lem:simplebijection}. 
\begin{theorem}
\label{thm:theta}
There exists a bijection $\theta: \Av(231|12) \rightarrow \Av(321|12)$. 
\end{theorem}
\begin{proof}
For visual aid, see Figure~\ref{fig:theta}. Let $(P,h_P)$ be a permutation from $\Av(231|12)$. We define $\theta$ to be $\lambda$ when restricted to $P_1$ and to be an identity map on $h_P$. Provided the gridding is unique for every $P \in \Av(231)$ and the resulting object is a permutation from $\Av(321|12)$, $\theta$ is a bijection. Every $\Av(231|12)$ griddable permutation has exactly one gridding of our choice -- pushing the gridline as right as possible. Given that $\lambda$ fixes the articulation point of $P$ and $h_P$ is unchanged under $\theta$, the gridline in the image of $\theta$ is as far right as possible -- making the gridding unique. Hence, $\theta: (P,h_P) \mapsto (\lambda(P),h_P)$ is a bijection from $\Av(231|12)$ to $\Av(321|12)$. 
\end{proof}

The next bijection acts in an analogous way to $\theta$, except this time the sequence on the RHS is increasing. Therefore, we additionally need to show that there always is a point on the RHS below an articulation point.

\begin{theorem}
\label{thm:psi}
There exists a bijection $\psi: \Av(231|21) \rightarrow \Av(321|21)$. 
\end{theorem}
\begin{proof}
Since the LHS classes in these juxtapositions are the same as those in Theorem~\ref{thm:theta}, the above proof carries over to this case as it is. One only needs to notice, additionally, that $\theta$ (and hence $\psi$) fixes the articulation point and acts as identity on the RHS sequence of points. This preserves the relative vertical position of the articulation point and any point on the RHS. Therefore, if there is a point on the RHS below the articulation point in a $P \in \Av(231|21)$, then there is a point on the RHS and below the articulation point of $\psi(P) \in \Av(321|21)$. 
\end{proof}

Unlike the bijections $\theta$ and $\psi$, the last bijection $\phi$ does not make use of Lemma~\ref{lem:simplebijection}. Instead, we reshuffle excursions in the Dyck paths and reverse the order of up steps in the middle excursion $\fM$. 

\begin{theorem}
\label{thm:phi}
There exists a bijection $\phi: \Av(312|21) \to \Av(312|12)$. 
\end{theorem}
\begin{proof}
The idea of this bijection is to ``reverse'' the Dyck path on the LHS of the juxtaposition in some sensible way --- this also leads to the sequence $h_P$ being reordered. Provided this is done in a reversible way, we are done. Let $P \in \Av(312|12)$ be $P = P_1P_2$. Let $P_1$ be represented by a Dyck path $\P$ and let $\P_1 \oplus \cdots \oplus \P_k$, $k\geq 1$, be the \emph{decomposition of $\P$ into excursions}. Recall from Section~\ref{sec:classC} the way $\Av(312|21)$ was enumerated. One of the excursions in $\P$, say $\P_i$, was described by part $\fM$ in the CFG. Construct a new Dyck path $\Q'$ from $\P$ by setting $\Q' = \P_{i+1}\oplus \cdots \oplus \P_{n} \oplus \P_i \oplus \P_1\oplus \cdots \oplus \P_{i-1}$. Recall that the part $\P_i$ is of the form $\fU\ldots\bU_1\ldots \fR$. We can, without interfering with the structure of the excursion $\P_i$ write the names of the up steps in the opposite direction to obtain $\Q_i$: a letter from $\{\fU, \bU_1, \bU\}$ at position $\ell$ in $\P_i$ will assume position $|\P_i|-\ell$ in the new $\Q_i$. This way, up steps and down staps remain preserved (the shapes of $\P_i$ and $\Q_i$ are the same). Given that $\P_i$ (hence also $\Q_i$) is an excursion, there is a 21 in $P_1$ to enclose the lowest (in the case of $\Q_i$ the highest) point on the RHS. Given the transformation from $\P$ to $\Q = \P_{i+1}\oplus \cdots \oplus \P_{n} \oplus \Q_i \oplus \P_1\oplus \cdots \oplus \P_{i-1}$, vertical steps in $\Q$ are in a different order than they originally were in $\P$. This results in the corresponding sequence $h_Q$ being a respective reshuffle of $h_P$. It is easy to check that the permutation $(Q,h_Q)$ is in the juxtaposition class $\Av(312|12)$. Given that the above transformation from $\Av(312|21)$ to $\Av(312|12)$ is unambiguous (i.e.~reversible), we conclude that 
$$\P_1 \oplus \cdots \oplus \P_k \mapsto \P_{i+1}\oplus \cdots \oplus \P_{n} \oplus \Q_i \oplus \P_1\oplus \cdots \oplus \P_{i-1}$$
gives rise to the bijection $\phi: \Av(312|21) \to \Av(312|12)$ such that $\phi: (P,h_P) \mapsto (Q,h_Q)$.
\end{proof}

Theorems~\ref{thm:theta}, \ref{thm:psi} and \ref{thm:phi} complete the information in Table~\ref{tab:bijections}.

\section{Conclusion}
\label{sec:conclusion}

Table~\ref{tab:bijections} enumerates the ``simplest'' grid classes that are not monotone. The next step could see Catalan class replaced by a more complicated one, or increased number of cells. If exact enumeration is beyond reach, it would be interesting to obtain information about the generating functions of such grid classes. For instance, consider a grid class formed by taking a monotone geometric grid class and replacing one cell by a Catalan class. Does every such grid class have an algebraic generating function? Can we say anything about generating functions of similar grid classes with a non-Catalan cell?

A problem not closely related to the exploration of grid classes, but interesting nevertheless, is finding a bijection between the juxtaposition class $\mathbf{C} = \Av(312|21)$ and the class $\Av(4312, 3142)$. They happen to be enumerated by the same sequence \href{http://oeis.org/A165538}{A165538} on \href{http://oeis.org}{OEIS} as we repeatedly mentioned before. See Albert, Atkinson and Vatter~\cite{albert2012inflations} for their enumeration of $\Av(4312, 3142)$.


\bibliographystyle{alpha}
\bibliography{juxt} 
\end{document}